\documentclass[nosumlimits,twoside]{amsart}
\usepackage{amsfonts}
\usepackage{amssymb}
\usepackage{graphicx}
\usepackage{amsmath}
\usepackage[ansinew]{inputenc}
\usepackage{color}
\usepackage{float}
\input xy
\newtheorem{thm}{Theorem}[section]
\newtheorem{defn}[thm]{Definition}
\newtheorem{ex}[thm]{Example}

\newtheorem{lem}[thm]{Lemma}
\newtheorem{pro}[thm]{Proposition}
\newtheorem{cor}[thm]{Corollary}
\newtheorem{rem}[thm]{Remark}
\newtheorem{rems}[thm]{Remarks}

\newtheorem{q}{Question}
\newtheorem{claim}{Claim}
\def\Z{{\mathbb Z}}
\def\N{{\mathbb N}}
\def\R{{\mathbb R}}

\def\T{{\mathbb T}}
\def\cont{\mathfrak c}

\begin{document}
\noindent
\title{
On the existence of  topologies compatible with a  group duality   with predetermined properties}

\thanks{The second and third authors thank the  financial support of the Spanish AEI and FEDER UE funds. Grant: MTM2016-79422-P. 
}
  \author{Tayomara Borsich}
\address{Departamento de Algebra, Geometr\'{\i}a y Topolog\'{\i}a,
Universidad Complutense de Madrid,
 28040 Madrid, Spain}
\email{tborsich@ucm.es}

\author{Xabier Dom\'inguez}
\address{Departamento de Matem\'aticas, Universidade da Coru\'na,
 Spain}
\email{xabier.dominguez@udc.es}

\author{Elena Mart\'in-Peinador}
\address{ Instituto de Matem\'atica Interdisciplinar y Departamento de Algebra, Geometr\'{\i}a y Topolog\'{\i}a,
 Universidad Complutense de Madrid,
 28040 Madrid, Spain}

\email{em\_peinador@mat.ucm.es}


\date{}


\maketitle
\begin{center}
	{\em Dedicated to Sergey Antonyan on his $65^{th}$ birthday. We deeply appreciate his friendship.}
\end{center}

\begin{abstract}
	The paper deals with group dualities. A group duality is simply a pair $(G, H)$ where $G$ is an abstract abelian group and $H$ a subgroup of characters defined on $G$. A group topology $\tau$ defined on $G$ is {\it compatible} with the group duality (also called dual pair) $(G, H)$  if $G$ equipped with $\tau$ has dual group $H$.
	
	  A topological group $(G, \tau)$ gives rise to the natural duality $(G, G^\wedge)$, where $G^\wedge$ stands for the group  of continuous characters on $G$. We     prove that  the existence of a  $g$-barrelled topology on $G$  compatible with the dual pair $(G, G^\wedge)$
	is equivalent to the semireflexivity in  Pontryagin's sense of the  group $G^\wedge$ endowed with the pointwise convergence topology $\sigma(G^\wedge, G)$.  We also deal with  $k$-group topologies. We prove that the existence of  $k$-group topologies on $G$  compatible with the duality $(G, G^\wedge)$  is determined by a sort of completeness property of its Bohr topology $\sigma (G, G^\wedge)$ (Theorem \ref{hay}). \\
	
\end{abstract}

For a topological abelian group $(G, \tau)$, denote by  $G^\wedge: = C Hom (G, \T)$  the group of all continuous characters on $G$.  The weak topology associated to $G^\wedge$ is  defined as the weakest topology on $G$ for which all the elements of $G^\wedge$ are continuous. It is a group topology which  will be denoted by $\tau^+$ (or by $\sigma (G, G^\wedge)$ if the duality $(G, G^\wedge)$ is the prevailing   point of view).  Clearly,  $\tau^+ \leq \tau$ and it is the bottom element in the duality $(G, G^\wedge)$.  By its relationship with the Bohr compactification of $(G, \tau)$, $\tau^+$  is called the Bohr topology of $(G, \tau)$.  It is  precompact and Hausdorff provided  $(G, \tau) $ has sufficiently many continuous characters. 
The question of when a  precompact and Hausdorff group topology on an abelian group is the Bohr topology corresponding to a locally compact group  has been considered in \cite{CTW},  in \cite{JG} and recently in \cite{HT}. The present paper was originated by a thorough reading of \cite{HT}. 

 More explicitly, the main  question in  \cite{HT} was:
 If  $(G, w)$ denotes  a totally bounded  abelian topological group (that is, precompact and Hausdorff),  is there a locally compact topology on $G$, say $\tau$, such that $\tau^+ = w$? 
 If such $\tau$ exists, it can be said in categorical language that $(G, w)$ is the Bohr reflection of $(G, \tau)$.
  The authors of \cite{HT}
   denote by $\mathcal B$ the class of all totally bounded abelian groups  which are the Bohr reflection of a locally compact group.
In the present  paper we  consider the question from another point of view.
As a matter of fact a precompact Hausdorff topological group $(G, w)$  is in $\mathcal B$ if  there is a locally compact topology in the duality $(G, G^\wedge)$, where  $G^\wedge $ denotes the character group of $(G, w)$.  Since in particular, every locally compact abelian group is $g$-barrelled, the question can be generalized to the following one:
\begin{q}
Let $(G, \tau)$ be an abelian  topological group. Under which conditions on $G$ or in $G^\wedge$ is there a $g$-barrelled topology in  the duality $(G, G^\wedge)$?
\end{q}

The $g$-barrelled groups were introduced  in \cite{CMT}.  In   Section \ref{main} we formulate their definition, and   we  obtain a necessary and sufficient condition for a  duality $(G, G^\wedge)$ to contain $g$-barrelled group topologies (Theorem \ref{existencia}).

 The existence of   dualities $(G, G^\wedge)$ {\bf without} $g$-barrelled topologies can  be  derived from deep  results of the  papers \cite{Auss1},  \cite{zn} and \cite{G}, where  the so called  ``Mackey problem for groups"  is solved in the negative. More precisely, the authors of those papers present examples of topological groups $(G, \tau)$ such that, the supremum  of the  family of all  the topologies on $G$ which are locally quasi-convex and  compatible with $\tau$ is  not compatible  with $\tau$.
  The Mackey topology on a topological group $(G, \tau)$  is defined as the maximum  -provided it exists- of all locally quasi-convex  topologies on $G$ compatible with $\tau$ (\cite{CMT}). The above  mentioned papers provide thus examples of locally quasi-convex    groups without a Mackey topology, which makes evident   the dissonance between the behaviour
 of  locally convex spaces, and that of  locally quasi-convex groups. We remind the reader that  the Mackey-Arens Theorem asserts that for a fixed locally convex space $(E, \rho)$, there exists a  maximum in the family of all  the topologies on $E$ that are  locally convex and  compatible with $\rho$.

 A $g$-barrelled {\bf locally quasi-convex} topology $\mu$ on a group $G$ is always the maximum of all the  topologies which are  locally quasi-convex and compatible with $\mu$ \cite[4.1]{CMT}.
  Thus, whenever the existence of a $g$-barrelled locally quasi-convex  topology in a fixed duality $(G, G^\wedge)$ can be guaranteed, it is unique  and
 it is  the  Mackey topology. Consequently, if the Mackey topology for a  topological group $(G, \tau)$ does not exist, the dual pair  $(G, G^\wedge)$ does not contain either a $g$-barrelled topology on $G$.  However, a Mackey topology may not be $g$-barrelled: these relationships, together with a grading of the Mackey property, are deeply analyzed in  \cite{JME}.
 In Theorem \ref{existencia} we give a necessary and sufficient condition for the existence of a $g$-barrelled topology in a  group duality.

   If in a given duality $(G, G^\wedge) $ there exists a locally quasi-convex  $g$-barrelled, non locally compact   topology $\nu$,  then $(G, \nu^+) \notin \mathcal B$. In this way a wealth of examples of groups  which are not in $\mathcal B$   can be obtained, a complement to  the results of \cite{HT}.

    The main results of the present paper are in   Section \ref{main}. Under the mild condition that a duality $(G, G^\wedge)$ is separated, we  prove (Theorem \ref{existencia}) that  the existence of a $g$-barrelled topology $\mu$ on $G$ such that $(G, \mu)^\wedge = G^\wedge$   is equivalent to the semireflexivity (in Pontryagin's sense) of the dual group $G^\wedge$ endowed with the pointwise convergence topology $\sigma (G^\wedge, G)$.
   If this holds,  $\mu$ is precisely the compact-open topology $\tau_\mathcal{K}$ on $G$ considered as the dual group of $(G^\wedge, \sigma(G^\wedge, G))$. Without requiring that $\tau_\mathcal{K}$ be compatible with the original topology of $G$, we  characterize when $(G, \tau_{\mathcal{K}})$ is a $g$-barrelled group in the duality that it generates (Theorem \ref{determined}).

  In Section \ref{kT-ext} we deal with the existence of  a $k$-group topology in a general  duality  $(G, G^\wedge) $.
  The $k$-groups, defined by Noble in \cite{noble}, constitute a class of abelian  topological groups that includes the locally compact abelian ones.  More generally,   all the topological groups that are $k$-spaces (in the ordinary sense of this term  for topological spaces) are $k$-groups.
   However there are $k$-groups  in the sense of Noble that are not $k$-spaces (See  \ref{prod lineas}).
    In  Section \ref{k-groups} we clarify these notions  and also recall  the  $k_\T$-groups introduced in \cite{TM}.
   	 The  $k_\T$-groups  are  relevant because of their connection with completeness. In \cite{BCMT} they appear while proving  that  the  Grothendieck Completeness Theorem, well known in the context of  locally convex spaces, does not admit a natural generalization  to locally quasi-convex groups. 
   	 
   We introduce   the notion of $k_\T$-extension of a precompact group topology.
   According to the property of being or not a $k_\T$-group, the family 
   $\mathcal P$ of precompact Hausdorff topologies on an abstract abelian group $G$  can be split into the two well differentiated subfamilies:

(I)  $\mathcal{P}_1 $ formed by  all those $w \in \mathcal P$ such that $(G, w)$ is a $k_\T$-group.  The elements $w \in \mathcal{P}_1 $ give rise to dualities $(G, G^\wedge)$ that contain at least one  $k$-group topology (Theorem \ref{hay}), which in turn can be locally compact, or metrizable or none of them, as shown in  Example \ref{prod lineas}.

(II)   $\mathcal{P}_2 $ formed by all those $w \in \mathcal P $ such that $(G, w)$ is {\bf not} a $k_\T$-group. The elements $w \in \mathcal{P}_2 $ produce dualities  $(G, G^\wedge)$ which do not contain a $k$-group topology (Theorem \ref{hay}).
   Nevertheless, the  $k_\T$-extension of each  $w \in \mathcal{P}_2$ is a  new precompact topology on $G$, which gives rise  to a group duality that {\bf contains}  $k$-group topologies (Theorem \ref{picture}). Further, $w$ and its $k_\T$-extension  produce the same   family of compact subsets.

In 	the last section we develop some  results about  the duality generated by a discrete group. This is a particular case of dualities which contain a $g$-barrelled topology $\tau$, such that $(G, \tau)$ has the Glicksberg property.

\section{Notation and Remarks} \label{NR}

		If $G$ is an abelian group, the set of  all homomorphisms from $G$ to $\T$ will be denoted by $ Hom (G,\T)$, where $\T$ is the unit complex circle. The elements of $Hom (G,\T)$ are called {\it characters} and $Hom (G,\T)$ has a group structure with respect to the pointwise operation. We shall write  $Hom_p (G, \T)$ to indicate that  $Hom (G,\T)$ is equipped with the pointwise convergence topology.
		
		A {\it group duality} is a pair $( G,H)$ where $G$ is an abelian group and $H$ is a subgroup of $Hom (G,\T)$. If $H$ separates points of $G$, the duality is said to be separated.
		
		If $(G, \tau)$ is a topological abelian group, its {\it dual group} or {\it character group} $G^\wedge:= CHom (G,\T)$  is the set of all continuous characters of $G$. It is a subgroup of $Hom (G,\T)$
		and if it   separates points of $G$, we say that $(G, \tau)$ is MAP (a shorthand for ``maximally almost periodic").
		
		Let $A\subseteq G$ and $B\subseteq G^\wedge$. The {\it polar } set of $A$ is defined by
		$$A^\triangleright=\{\chi\in G^\wedge~:~\chi(x)\in\T_+~~ \forall x\in A\}$$
		and the {inverse polar } of $B$ is defined by
		$$B^\triangleleft=\{x\in G~:~\chi(x)\in\T_+~~ \forall \chi\in B\}$$
		where $\T_+:=\{e^{2\pi it} : t\in [-\frac{1}{4}, \frac{1}{4}]\}$.
		
		A subset $A \subseteq G$ is {\it quasi-convex} if for every $x \in G\setminus A$ there exists an element $\phi \in A^{\triangleright}$ such that $\phi(x) \notin \T_+$.
		The {\it quasi-convex hull} of a  subset $M\subseteq G$ is the smallest quasi-convex subset of $G$ that contains $M$. It is straightforward to prove that it coincides with  $M^{\triangleright\triangleleft}$; in particular $M$ is quasi-convex if and only if $M=M^{\triangleright\triangleleft}$.  The topological group $(G, \tau) $ is said to be {\it locally quasi-convex } if it admits a basis of neighborhoods of zero formed by quasi-convex subsets.

		For a topological group $(G, \tau)$, the finest among all the locally quasi-convex topologies on $G$ coarser than $\tau$ is the {\it locally quasi-convex modification of  $\tau$}. It
		will be denoted by  $\mathcal Q \tau$ and has as a basis of 0-neighborhoods the family  $ \mathcal{B}  = \{U^{\triangleright\triangleleft},\;   U\in {\mathcal N} \}$, where ${\mathcal N}$ stands for the $\tau$-neighborhood system of the neutral element.
		
		Let  $(G, \tau)$ be a topological group. A subset $S\subseteq G^\wedge$ is equicontinuous with respect to $\tau$ if and only if $S\subseteq U^\triangleright$ for some $\tau$-neighborhood of zero $U$ in $G$.   This is a simpler 
		formulation, for abelian topological groups and families of continuous characters, of the well-known notion of equicontinuous set  of  mappings   in the context of uniform spaces.  
		
			For an abelian group $G$ and a subgroup of characters on $G$, say $ \mathcal L \subset Hom (G, \T)$, $\sigma(G, \mathcal L)$ will denote the weak topology on $G$ with respect to the family $\mathcal L$.
		If we start with a topological group $(G, \tau)$,    we  will replace $\sigma(G, G^\wedge)$ by  $\tau^+$, whenever this symbol  is easier to handle.
		
		Symmetrically, $\sigma(G^\wedge, G)$ denotes the weak  topology on $G^\wedge$ with respect to the evaluation mappings corresponding to the elements of $G$.

	If the context is clear,  $G^\wedge$ also denotes the character group {\bf endowed with the compact-open topology}. The latter is the natural topology to deal with reflexivity in Pontryagin's sense, so we often use the term {\it Pontryagin dual} to underline that   $G^\wedge$ carries the compact-open topology.
	
	If $\tau_1$ and $\tau_2$ are group topologies on $G$ we will say that they are {\it compatible} if $(G, \tau_1)^\wedge = (G, \tau_2)^\wedge$. For a dual pair $(G, H)$ a topology $\tau$ on $G$  is said to be {\it compatible with the duality $(G,H)$} or simply {\it to be  in  the duality} $(G, H)$ if $(G, \tau)^\wedge = H$.

	In the sequel it will be implicitly understood that all the groups considered are abelian.
\section{A short trip through $k$-spaces, $k$-groups and $k_\T$-groups}\label{k-groups}
Although the notion of $k$-space is well known, there is no uniformity in the literature whether it might be defined in the framework of Hausdorff spaces or simply in the context of topological spaces.   Therefore  we  make precise our starting point and the   properties which require further assumptions.

A topology $\tau$ on a set $X$ is called a $k${\it-topology} if the following condition holds:

\begin{center}
	Whenever $H \subseteq X$ is such that $H \cap K$ is $\tau$-closed in $K$
	for every $\tau$-compact $K$,\\ then $H$ is closed in $\tau$.
\end{center}

  If $\tau $ is a $k$-topology on $X$, the pair $(X, \tau) $ is called a $k$-space.  As a matter of fact, there is  a  $k$-topology associated to each topology $\tau$ on a set $X$. It  is commonly called the $k$-refinement of $\tau$ (in the literature, also the $k$-extension), and it is defined by its family of closed sets  as follows:

   \begin{defn}

   {\em	 The {\it $k$-refinement} $k(\tau)$ of  a topology  $\tau$ on  a set  $X$ is  defined by:
 $C \subset X $ is  closed in  $k(\tau)$ if $C \cap K$ is $\tau$-closed in $K$, for every $\tau$-compact subset $K \subset X$.  The pair $(X, k(\tau))$ is also called the $k$-refinement of $(X, \tau )$. } 
  \end{defn}

Clearly $k(\tau)$   is well defined, 
it is a $k$-topology  and $ \tau \leq k(\tau)$. The equality holds if $\tau$ is already a $k$-topology.

  The $k$-refinement $k(\tau)$ of a topology  $\tau$ on a set $X$
    gives rise to the same compact subsets as $\tau$.
       With the additional  assumption that  $(X, \tau) $ is Hausdorff, it holds that   $k(\tau) $ is the finest topology on $X$ with  this property. For this reason some authors  define the $k$-refinement only  for a {\bf Hausdorff} topology $\tau$.  In order to avoid confusion,
        we provide a proof of these facts.

\begin{pro}\label{kspaces}
	Let $(X, \tau)$ be a topological space and let $k(\tau)$ be the $k$-refinement of $\tau$. Then the following statements hold:
	\begin{itemize}
		\item [(1)]  $\tau $ and $k(\tau)$ give rise to the same compact subsets.
		\item [(2)]  $\tau $ and $k(\tau)$  induce the same topology on any compact $K \subset X$.
		\item[(3)] For every topological space $(Y, \mu)$, a  function $f: X \to Y$  such that   the restriction  $f_{|K}$ to any compact subset $K \subset X$ is continuous,   is necessarily continuous with respect to $k(\tau)$.
	\end{itemize}
Furthermore, $k(\tau)$ is the finest among the topologies on $X$ which 
 satisfy the condition of (2) or  (3).
If $(X, \tau)$ is Hausdorff,  $k(\tau)$ is also the finest among the topologies on $X$ which have  the same  compact subsets as $\tau$. 

\end{pro}

\begin{proof}
(1)	Since $\tau \leq k(\tau)$, every $k(\tau)$-compact is $\tau$-compact.\\
For the converse, fix a $\tau$-compact subset  $L \subset X$. In order to show that $L$ remains $k(\tau)$-compact, take a cover $\mathcal U \subset k(\tau)$ of $L$.  By the definition of the topology $k(\tau)$, for every $U \in \mathcal U$ the intersection $U\cap L$ is open in $L$. Since $L$ is compact, the cover $\{U \cap L: U \in \mathcal U \}$ has a finite subcover, and so does $\mathcal U$.

     (2) follows from  (1) and  the definition of $k(\tau)$.
     
     Finally (3) follows  from
 the equality $f^{-1}(D) \cap K = (f_{|K})^{-1}(D)$ for any $D \subset Y$ and $K \subset X$ and  the definition of $k(\tau)$.
 
 It is straightforward to prove that $k(\tau)$ is the finest topology on $X$ satisfying (2) or (3).

In order  to prove the last assertion, assume that $\mu$ is a topology on $X$ which gives rise to the same  compact subsets as  $\tau$.
 If $C \subset X$ is $\mu $-closed, for every $K \subset X$ compact,  $C \cap K$ is $\mu$-compact, and by the assumption, it is also  $\tau$-compact. Since $(X, \tau)$ is Hausdorff, $C \cap K$ is    $\tau$-closed. This holds for any $\tau$-compact subset $K$, therefore  $C$ is $k(\tau)$-closed and  $\mu \leq  k(\tau)$. \\
  The requirement  on  the space $X$ to be   Hausdorff is used when claiming that every  compact subset is  closed. 
 	\end{proof}

  Two topologies $\tau_1$ and $\tau_2$ on a set $X$  for which the families of compact subsets of $X$ coincide 
  may not induce the same topology on a compact subset of both, as the following example shows.

  {\bf Example}. Let $X = \{1/n, n \in \N \} \cup \{0 \}$, let $ \tau_1$ be the topology on $X$ induced by the Euclidean of $\R$   and  $\tau_2$ the topology whose open sets are all the subsets of $X$ that do not contain $\{0 \}$, together with the total set $X$. Clearly $\tau_1$ and $\tau_2$ have the same family of compact subsets: namely, every finite $F \subset X$, and every subset that contains $\{0 \}$.  However they do not induce the same topology on the compact subset $ \{1/2n, n\in \N \} \cup \{0 \}$. 

   On the other way round, if the assumption is  that both topologies induce the same topology on the compact subsets of one of them, say $\tau_1$, then   $\mathcal K_1 \subset \mathcal K_2$, where $ \mathcal K_i$ are the respective families of compact subsets.


  \begin{lem}\label{cptosiguales}
   Let   $\mathcal F = \{ \tau_i, i \in I \}$  be a family of  Hausdorff topologies on a set   $X$ which give rise to the same compact subsets. The elements of $ \mathcal F$  induce the same topology on the common  compact subsets of $X$. If $\tau_1 $ is  the supremum of $\mathcal F$, $\tau_1$  has also the same family of compact subsets  and $(\tau_1)_{|K} = (\tau_i)_{|K} $ for every  compact $ K \subset X$.
  \end{lem}

   \begin{proof}
    In order to prove the first assertion,	fix  a  $\tau_i$-compact subset  $K$. Since $\tau_i$  is Hausdorff, any  $C\subset K$ is $\tau_i$-closed if and only if it is $\tau_i$-compact.  
   	Thus, by the assumption,  the $\tau_i$-closed subsets of $K$ are the same for every $i \in I$. \\
   	Let us see now that  every  $\tau_i$-compact is also   $\tau_1$-compact. As above, let $K \subset X$  be  $\tau_i$-compact. Pick a net
   	 $S: = \{x_j, j \in J \}$  in  $K$. It has a $\tau_i$-convergent subnet, thus     without loss of generality we can assume   that there exists $ x \in K$ such that 	 $x_j  \stackrel{\tau_i}{\longrightarrow} x$, for every $i \in I$.\\
   	  Let us prove that $S$  also converges to $x$ in $\tau_1$.  A basic  $\tau_1$-neighborhood   of $x$ has the form   $V = \cap_{m=1}^n V_{i_m }$, where $V_{i_m}$ is a neighborhood of $x$ in $\tau_{i_m}$. Since   	 $x_j  \stackrel{\tau_{i_m}}{\longrightarrow},  x$,  for $m= 1, \dots, n$,  the net is eventually in $V_{i_1}, \dots, V_{i_n}$. Now  $J$ is a directed set, therefore  the net is also eventually in $V = \cap_{m=1}^n V_{i_m }$. As $V$ was a basic arbitrary $\tau_1$-neighborhood of $x$,  it follows that  $x_j  \stackrel{\tau_1}{\longrightarrow} x$.
   	  \end{proof}

 The $k$-refinement  of a group topology may not be a group topology (Example \ref{prod lineas}).
 Around the 70's Noble defined the $k$-groups, a  notion  weaker than that of a  $k$-space in the context of topological groups. The  $k_\T$-groups were defined in \cite{TM} in the context of abelian topological groups. 
 For the reader's convenience we state both definitions:

 \begin{defn}
 {\em	A topological group $(G, \tau)$ is} a $ k$-group {\em if for every topological group   $(H, \mu)$  and every homomorphism $f: G \to H$ the following holds:
 		
 If	$f_{|K}$ is continuous for any compact  $K \subset G$, then $f$ is continuous.}
 \end{defn}
  For further use we  write the following  property, whose proof can be seen in \cite[1.1]{noble}:
  \begin{lem}\label{k-group}
  A  topological group  $(G, \tau)$ is a $k$-group iff $\tau$ is the finest {\bf group}  topology among all  those  that induce the same  as $\tau$ on the $\tau$-compact subsets.	
  \end{lem}

 To  each  topological group 
  there corresponds  a $k$-group structure,  defined as follows:

 Where  $(G, \tau)$ is a Hausdorff topological group, let $k_g(\tau)$ be the   finest  of all the group topologies on $G$ that coincide with $\tau$  on every $\tau$-compact subset $K \subset G$.
 Clearly $(G, k_g(\tau))$ is a $k$-group, which might be called the {\it  $k$-group modification of} $(G, \tau)$. Also the topology $k_g(\tau)$ is called the $k$-group modification of $\tau$.

 Observe that if  a topological group $(G, \tau)$ is a $k$-space, it is in particular a $k$-group. However there are $k$-groups which are not $k$-spaces as shown in  Example  \ref{prod lineas}

 \begin{defn}
 {\em	A topological group is} a $k_\T$-group {\em if every character $f: G \to \T$ such that its restriction $f_{|K}$  to any  compact $K \subset G$ is continuous, must be continuous.}
 \end{defn}

 Every $k$-group is a $k_\T$-group, as a consequence of their definitions. The converse does not hold as shown in (v) of  the following example, 
 partially stated in  \cite[6.1.3]{TM}.

\begin{ex}\label{Ex Banach}
	{\bf A $k_\T$-group which is not a $k$-group.}\\
	{\em  Let $E$ be an infinite-dimensional  Banach space.
		 Denote by $\sigma(E, E^*)$ and  $\sigma (E, E^\wedge)$  the weak topologies on $E$ corresponding respectively to the space $E^*$ of continuous linear forms,   and  to the group $E^\wedge$ of continuous characters.}
	\begin{itemize}
		\item[(i)] 
 $\sigma (E, E^\wedge) < \sigma(E, E^*) $. 	{\em  The inequality is strict, but both topologies  give rise to the same character group, that is:}
		$$	(E, \sigma (E, E^*))^\wedge =  (E, \sigma (E, E^\wedge))^\wedge =  E^\wedge $$ 
			\item[(ii)] {\em    $ (E, \sigma (E, E^*)) $ is a $k_\T$-group.}
	
			\item[ (iii)] {\em  $\sigma(E, E^*)$ and   $\sigma(E, E^\wedge)$ give rise to the same family of compact subsets of $E$.} 
			\item[(iv)]  $(E, \sigma(E, E^\wedge)) $ {\em is also a $k_\T$-group.}
			\item[(v)]  {\em If $ E: = l_1$ is  the classical Banach  space, $G: = (l_1, \sigma (l_1, l_{\infty})) $   is  a  $k_\T$- group, which is not a $k$-group.}
	\end{itemize}
\end{ex}

\begin{proof} (i) By \cite[Lemma 1]{smith} the exponential mapping from  $E^*$ to $E^\wedge$ is a group isomorphism. Further, the same  statement  holds for any locally convex space $X$, with similar meaning for $X^*$ and $X^\wedge$. In particular, if $X: = (E, \sigma(E, E^*))$, the continuous linear forms on $X$ are  the elements of $E^*$. Thus, the continuous characters on $X$ have the form $\rho \phi$, where $\rho: \R \to \T$ is defined by $t \mapsto e^{2\pi it}$ and $\phi \in E^*$. This proves that $	(E, \sigma (E, E^*))^\wedge = E^\wedge$ and therefore $\sigma (E, E^\wedge)$ is the  Bohr topology on $E$ associated  to $ \sigma (E, E^*) $. Consequently, $	(E, \sigma (E, E^*))^\wedge =  (E, \sigma (E, E^\wedge))^\wedge$. \\ 
	   It also holds  $\sigma (E, E^\wedge) < \sigma(E, E^*) $, and they are distinct since the first one is precompact and the second one is a vector topology. \\
	   	(ii) Let  $f: X \to \T$ be a group homomorphism such that   $f_{|K} $ is continuous for any  $\sigma (E, E^*)$-compact $K  \subset E $. If $\tau$ denotes the norm topology on $E$, from $\sigma(E, E^*) < \tau$ we obtain that $f_{|L} $ is $\tau$-continuous for every $L \subset E$ which is  $\tau$-compact.   Since  $(E, \tau)$   is a $k$-space,
	   	  $f: (E, \tau) \to \T $ is continuous, that is $f \in E^\wedge$. From (i) we obtain that $f \in (E, \sigma (E, E^*))^\wedge$, 
	   	   and therefore $X$ is a  $k_\T$-group. \\
	(iii) is is proved in  \cite{RT}.\\
	 (iv) is a consequence of (i), (ii), (iii) and the first part of Lemma \ref{cptosiguales}.\\
	 (v) By (ii), $G$ is a $k_\T$-group. Taking into account that  the $\sigma (l_1, l_{\infty})$-compact subsets  of $l_1$ coincide with the norm-compact subsets, and that $l_1$ is a $k$-space, we obtain that the norm topology is the $k$-refinement (clearly, also the $k$-group modification) of $\sigma (l_1, l_{\infty})$. By Lemmas \ref{k-group} and \ref{cptosiguales},  $G =(l_1, \sigma (l_1, l_{\infty}))$ is not a $k$-group.
	 \end{proof}

From Example \ref{Ex Banach}, it is clear  that for  a topological group $(G, \tau)$ there might exist several compatible $k_\T$-group topologies on $G$ with the same compact subsets as $\tau$.  On the other hand,  there is a unique $k$-group topology whose compact subsets are precisely  the $\tau$-compact subsets, namely $k_g(\tau)$.   We will see below that the $k$-group modification   $k_g(\tau)$ is compatible with $\tau$ if and only if $(G, \tau^+)$ is a $k_\T$-group.

For two group topologies on a group $G$,  the property of ``giving rise to the same compact subsets" is in some sense complementary to that of  ``being compatible", as expressed in the following proposition whose easy proof we omit.

\begin{pro}\label{contraste}
	Let $G$ be an abelian group and let $\tau_1, ~\tau_2$ be group topologies on $G$ such that $\tau_1 \leq \tau_2$.
	 \begin{itemize}
	 	\item[(1)] Assume $\tau_1$ and $\tau_2$ are compatible topologies.\\ If $(G, \tau_2)$ is a $k_\T$-group, then $(G, \tau_1)$ is also a $k_\T$-group.
	 	\item[(2)] Assume $\tau_1$ and $\tau_2$  give rise to the same family of compact subsets and both are Hausdorff. \\If $(G, \tau_1 )$ is a $k_\T$-group, then also
	 	$(G, \tau_2)$ is a $k_\T$-group.
	 \end{itemize}
 In particular, if there is a $k_\T$-group topology in $G$, compatible with  a group  duality $(G, H)$, then $(G, \sigma(G, H))$ is a $k_\T$-group.
\end{pro}
\begin{cor}
If a topological group $(G, \tau)$ has the Glicksberg property (i.e., the $\tau$-compact  and the $\tau^+$-compact subsets of $G$ coincide), then all the  group topologies on $G$ which lie between $\tau^+$ and $\tau$ are simultaneously $k_\T$-group topologies or else none of them is a $k_\T$-group topology.
\end{cor}
Next, we present an example of a $k$-group which is not a $k$-space. Its $k$-refinement is not even a topological  group.
 \begin{ex}\label{prod lineas}
	{\bf A $k$-group whose  $k$-refinement is not a topological group.}

 {\bf Claim 1.} The product $\R^\R$ of $\cont$ real lines is not a $k$-space. Nevertheless, it  is a $k$-group.
	
\end{ex}

\begin{proof}
The proof of the first claim can be developed through the following  hint provided in Kelley's book \cite{kelley}.	Let $A \subset \R^\R$ be defined by:
	
	$	x=(x_r)_{r \in \R} \in A \Leftrightarrow $ there exists $m \in \Z$ and $F \subset \R$ with $|F| \leq m$ and $m \geq 1$,  such that:
	\begin{equation*}
	\left\{
	\begin{array}{ll}
	1) ~~  x_r = m, ~ ~\forall r \in \R \setminus F\\
	
	2) ~~ x_r =0 , ~ ~{\rm if~} ~ r\in F  ~ 
	
	\end{array}
	\right.
	\end{equation*}

	It is easy to prove that $\overline{A} = A\cup\{0\}$, therefore $A$ is not closed. However, for every compact $K \subset \R^\R$, $A\cap K$ is closed in $K$. Thus $\R^\R$  is not a $k$-space.
	
	The fact that it is a $k$-group follows from \cite[Theorem 5.7]{N2}, where it is proved that the product of $k$-groups is a $k$-group. Obviously, each factor $\R$ is a $k$-space, therefore also a $k$-group.
\end{proof}

{\bf Claim 2.} {\em	Denote by $\pi$ the product topology on $\R^\R$, and let $k(\pi)$  be its $k$-refinement. Then $(\R^\R, k(\pi))$ is not a topological group.}

\begin{proof}
	As pointed out in Claim 1,
$\R^\R$ endowed with the product topology $\pi$ is not a $k$-space. If its $k$-refinement $k(\pi)$ were a group   topology,  $(	\R^\R, k(\pi))$  would be a $k$-group, with the same compact subsets as 	$(\R^\R, \pi)$. Since $(\R^\R, \pi)$ is already a  $k$-group,
		 it must be   $k(\pi) = \pi$. This contradicts the fact that $\pi$ is not a $k$-space topology and  therefore $\pi \neq k(\pi)$.
\end{proof}
\section{The  $k_\T$-extension of a precompact  topology}   \label{kT-ext}
   The  goal of this section is to determine conditions under which  the existence of   $k$-group topologies in a fixed  duality $(G, G^\wedge)$ can be guaranteed. The  $k_\T$-groups appear in this context  because of the following result:
   \begin{lem}\cite[6.1.5]{TM}\label{montse}
   	A topological group $(G, \tau)$ is a $k_\T$-group if and only if $k_g(\tau)$ is compatible with $\tau$, where $k_g(\tau)$ is  the $k$-group modification of $\tau$.
   \end{lem}
\begin{proof}  
	Let $(G, \tau)$ be a  a $k_\T$-group. Then $\tau \leq k_g(\tau) $ implies  $(G, \tau)^\wedge \subset (G, k_g(\tau))^\wedge$.\\
	For	the other inclusion, let $f: G \to \T$ be in $(G, k_g(\tau))^\wedge$.  From the fact that $\tau \le k_g(\tau) \le  k(\tau) $ and Proposition \ref{kspaces} it follows that $\tau$ and $k_g(\tau) $ admit the same family of compact sets and induce the same topology on them. Thus, $f_{|K}$ is $\tau$-continuous for every $\tau$-compact $K \subset G$.   Since $(G, \tau)$  is a $k_\T$-group,   $f$ is $\tau$-continuous and $f \in (G, \tau)^\wedge $.\\
	Conversely, assume $(G, \tau)^\wedge = (G, k_g(\tau))^\wedge$. Let $f: G \to \T$ be a homomorphism such that 
$	f_{|K} $ is $\tau_{|K}$-continuous for every compact $K \subset G$. By the definition of $k_g(\tau)$, $f$ must be $k_g(\tau)$-continuous. 
Thus,  $ f \in (G, k_g(\tau))^\wedge = (G, \tau)^\wedge$ and $(G, \tau)$ is a $k_\T$-group.
	\end{proof}
\begin{rem} \label{bohr}
 {\em If a topological group $(G, \tau)$ is a $k_\T$-group, then also $(G, \tau^+)$ is a $k_\T$-group. This derives from (1) of Proposition \ref{contraste},  and the fact that $\tau^+ $ is always compatible with $\tau$.} 

\end{rem}
   
  Thus, by the previous remark, the bottom topology $\sigma(G, G^\wedge)$ determines the existence of $k_\T$-topologies in a fixed duality $ (G, G^\wedge)$.    
   Let us see next that it also determines the existence of $k$-group topologies:

   \begin{thm}\label{hay}
   	Let $(G, \tau)$ be a topological group, and $G^\wedge$ its character group. The following statements are equivalent:
   	\begin{itemize}
   		\item [(1)]
   		There is at least one $k$-group topology on $G$ compatible with  the duality $(G, G^\wedge)$.
   		\item[(2)]
   		 $\tau^+ $ is a $k_{\T}$-group topology.
   	\end{itemize} 	
   \end{thm}
\begin{proof}
	
(1) $\Rightarrow$ (2) Assume that  $\mu$ is a $k$-group topology in the duality $(G, G^\wedge)$. This means that $(G, \mu)^\wedge = G^\wedge$. In particular, $ \mu$ is a $k_\T$-group  topology and by (1) of Proposition \ref{contraste}, $\mu^+$ is also a  $k_\T$-group  topology. Since $\mu$ is compatible with $\tau$,   $\tau^+ = \mu^+$ and  the assertion follows.\\
(2) $\Rightarrow$ (1)  Conversely, if $(G, \tau^+)$ is a $k_\T$-group, by Lemma \ref{montse}  we obtain that  $k_g(\tau^+)$ is a $k$-group topology  on $G$ compatible with $\tau^+$. In other words,  $k_g(\tau^+)$ is in the duality $(G, G^\wedge)$.
	\end{proof}
\begin{rems}
	
\begin{itemize}
	\item[(i)]	
{\em   Observe that in  a duality $(G, G^\wedge)$
	there might be  several compatible  $k$-group topologies.
	For instance,  in the duality $(c_0(\T), \Z^{(\N)})$ presented and studied in \cite{DiMT}, the maximum and the minimum of all the locally quasi-convex compatible topologies are both metrizable, thus both are compatible locally quasi-convex  $k$-group topologies.	
	\item [(ii)]
	There is at most one {\bf complete} metrizable locally quasi-convex topology on a group $G$ with a fixed dual group $G^\wedge$. As proved in \cite{CMT},  such a topology is the Mackey topology on $G$  for the corresponding  duality $(G, G^\wedge)$.

\item[(iii)]
 For a fixed duality $(G, G^\wedge)$ let $\mu$ be a compatible  topology on $G$, and let $\mathcal K_{\mu}$ be the set of $\mu$-compact subsets of $G$. Assign to   $\mathcal K_{\mu}$ the $k$-group topology it generates on $G$, say $k_g(\mu)$. Then, by Lemma \ref{montse},  $k_g(\mu)$ is compatible with the duality $(G, G^\wedge)$, if and only if   $(G, \mu)$ is a $k_\T$-group.
 Thus,  the set of $k$-group topologies compatible with the duality $(G, G^\wedge)$ is in 1-1 correspondence with the set of families  $\mathcal K{_\mu}$ for  $\mu$ a $k_\T$-group topology in the dual pair $(G, G^\wedge)$.}
\end{itemize}
\end{rems} 

 Let us denote by  $\mathcal P$ the family   of precompact topologies on an abelian abstract group $G$, and by $\mathcal P_1$   the subfamily   of those  $w \in  \mathcal P$ such that   $(G, w)$ is  a $k_\T$-group. As proved  in Theorem \ref{hay}, the elements  $w \in \mathcal P_2 := \mathcal P \setminus \mathcal P_1 $ produce dualities $(G, G^\wedge)$ which do not  contain $k$-group topologies. We next define  the $k_\T$-extension of     $w \in  \mathcal P_2$,  a sort of associated  precompact topology  which gives rise  to a new duality which contains $k$-group topologies.

{\bf Notation. }
Denote by $\mathcal M$  the family of $w$-compact subsets of  a precompact group  $(G, w)$. 
Let  $\mathcal{H} $ be the set of all the characters $ f \in Hom(G, \T)$ such that 
$f_{|K} ~$ is  $ w$-continuous, $ \forall K \in \mathcal{M}$

\begin{defn}\label{kT-def}
{\em Let $(G, w)$ be a   precompact group. 
 The weak topology  on  $G$ relative to $\mathcal{H}$,  henceforth  denoted     $\tau_{\mathcal H}$ will be called}  the  $k_{\T} $-extension of  $w$.  
\end{defn}

 Clearly $\tau_{\mathcal H}$ is a     precompact topology on $G$, and  
$(G, \tau_{\mathcal H})^\wedge = \mathcal H$ by  \cite[3.7]{CMT}.
Consequently,  $G^\wedge: = (G,w)^\wedge \subset \mathcal H$, and  $w \leq \tau_{\mathcal H}$. Observe that  the equality   $\tau_{\mathcal H} = w$  implies that   $(G, w)$ is already a $k_{\T} $-group.

{\bf Some  properties of the  $k_{\T} $-extension. }

\begin{pro}\label{propiedades}

Let $(G, w)$ be a precompact  group and let $\tau_{\mathcal H}$ be the
 $k_{\T} $-extension of $w$. Then,
\begin{itemize}
	\item[(1)]
	 $\tau_{\mathcal H}$ and $w$ give rise to the same family $\mathcal M$ of compact subsets of $G$. Further,  $w_{|K} = (\tau_{\mathcal H})_{|K} $, for all $K \in \mathcal K$.
	 
\item[(2)]
$\tau_{\mathcal H}$ is a  $k_\T$-group topology.
\item[(3)]
 $\tau_{\mathcal H}$ is the maximum    in the family of all precompact topologies on $G$ that coincide with $w$ in the $w$-compact subsets of $G$.
 	\item[(4)]
 	If $w$ is Hausdorff, $\tau_{\mathcal H}$ is the maximum in the family of all precompact  Hausdorff topologies on $G$ with the same compact subsets  as  $w$.
\end{itemize}
\end{pro}
\begin{proof}
(1) From the fact $w \leq \tau_{\mathcal H}$, we only need to prove that a fixed $w$-compact subset $K \subset G$ is also 	$\tau_{\mathcal H}$-compact.  To this end, pick a net  $S: = \{x_i, i \in I \}$   with range in $K$. Since $K$ is $w$-compact, $S$ has a $w$-convergent subnet. Without loss of generality, assume directly that $x_i  \stackrel{w}{\longrightarrow}  x$.  For every $f \in \mathcal H$ it holds $ f(x_i)
{\longrightarrow} f(x)  $ in $\T$. Therefore, taking into account that  $ (G,\tau_{\mathcal H})$ is precompact with dual $\mathcal H$,   $x_i  \stackrel{\tau_{\mathcal H}}{\longrightarrow} x$. Thus, $K$ is also $\tau_{\mathcal H}$-compact and further induces on $K$ the same topology as $w$.

(2) This is obvious from (1) together with the fact that $(G, \tau_{\mathcal H})^\wedge = \mathcal H$.

(3) Assume now that $u$ is a precompact topology in $G$ such that $u_{|K} = w_{|K}$ for every $w$-compact  $K\subset G$. If $\mathcal L = (G, u)^\wedge$, every $f \in \mathcal L$ is  clearly in $\mathcal H$. Therefore,
  $ u \leq \tau_{\mathcal H}$.

  (4)   The assertion follows from  (3) and   Lemma \ref{cptosiguales}.

	\end{proof}

 For a topological group $(G, \tau)$ such that  $\tau_{\mathcal H} \neq \tau^+$, the duality  $(G, G^\wedge)$ does not admit any $k$-group topology, as seen in Theorem \ref{hay}. However,  $\tau_{\mathcal H} $ gives rise to  a sort of ``extended duality"  which improves this lack, as shown next.
\begin{thm} \label{picture}
 	Let  $(G, \tau)$  be a  topological group, and $\tau_{\mathcal H}$ the $k_\T$-extension of $\tau^+$. Then,  $\tau_{\mathcal H}$  is   a $k_\T$-group topology on $G$ compatible with the duality    $(G, \mathcal H) $. Further,  $\tau_{\mathcal H}$ and $\tau^+$ give rise to the same family of compact subsets of $G$.

\end{thm}
\begin{proof}
	By (2)  of Proposition \ref{propiedades}, $\tau_{\mathcal H}$ is a
	$k_\T$-group topology.  Since 
	 $(G, \tau_{\mathcal H})^\wedge = \mathcal H$, according to  Lemma \ref{montse},  $k_g(\tau_{\mathcal H})$ is a $k$-group topology compatible with  the duality $(G, \mathcal H)$.  The last assertion is proved by (1) in Proposition \ref{propiedades}.
	\end{proof}

 The family $\mathcal P$ of precompact topologies on a fixed abelian group $G$ offers  the following picture:
\begin{itemize}
	\item  If  $(G, w)$ is a    $k_\T$-group (that is, $w \in \mathcal P_1$), there exist $k$-group topologies  on $G$ compatible with  the duality $(G, G^\wedge)$, where $G^\wedge = (G, w)^\wedge$. They  might be locally compact or metrizable or none of them, as shown by the example $\R^\R$  (\ref{prod lineas}).
	\item
	If $(G, w)$ is    not a $k_\T$-group (that is,  $w \in \mathcal P_2$), the   $k_\T$-extension of $w$, which we have called  $\tau_{\mathcal H}$,  gives rise to a {\bf new} duality $(G, \mathcal{H})$, with   $k$-group topologies. The latter are not compatible with $w$. In fact, if $\lambda$ is one of them,  $(G, \lambda)^\wedge = \mathcal H \neq (G, w)^\wedge $.
\end{itemize}

\begin{pro}\label{complecion}
	Let $(G, w)$ be a precompact group with  $w \in \mathcal P_2$, and let $\tau_{\mathcal H}$ be the $k_\T$- extension of $w$. Then, the Pontryagin dual of $(G, \tau_{\mathcal H})$ is complete and contains $(G, w)^\wedge$ as a topological subgroup.  However, $(G, \tau_{\mathcal H})^\wedge$ may not be  the completion of $(G, w)^\wedge$.
\end{pro}
\begin{proof}
	First, let us state that the Pontryagin dual of a $k_\T$-group is complete \cite[6.1.6]{TM}, therefore $(G, \tau_{\mathcal H})^\wedge$ is complete. Clearly,  $(G, w)^\wedge \subset (G, \tau_{\mathcal H})^\wedge$. Since   $w$ and $\tau_{\mathcal H}$ give rise to  the same compact subsets, the  dual group $(G, w)^\wedge$ is a topological subgroup of $(G, \tau_{\mathcal H})^\wedge$, where both  are considered with the compact-open topology.
	
	We must prove now that $ G^\wedge: = (G, w)^\wedge$ is not necessarily dense in  $(G, \tau_{\mathcal H})^\wedge =\mathcal H$   endowed with the compact-open topology.  The following example provides a proof of it.
	\end{proof}
\begin{ex}

{\bf A precompact group $G$ with complete Pontryagin dual, which is not  a $k_\T$-group.}

	 {\em We describe first some auxiliary tools before giving explicitly the example in Claim 1.\\
	 	
	 	Let $L^2[0, 1]$ be the Hilbert space of   square integrable functions on $[0,1]$, and let  $L: = L^2_\Z[0, 1] \subset L^2[0,1]$  be the subgroup  formed by all the almost everywhere integer valued functions, equipped with the induced topology.
	 This group is  considered  in \cite[Section 11]{Auss0}, where a remarkable  proof of the following fact is given:
	
	 (*) The Pontryagin dual of $L$  is topologically isomorphic to the  dual of $L^2[0,1]$, say  $$L^\wedge \approx (L^2[0,1])^\wedge$$ 
	 through the restriction mapping.
	
	  We  state some other  properties of $L$ and $L^\wedge$   needed for our argument: \\
	  \begin{itemize}
	  	\item[ (1)]
	  $L$ is not Pontryagin reflexive. In fact, the natural mapping from $L \to L^{\wedge \wedge}$ is a non-surjective embedding. This derives from (*) and from the Pontryagin reflexivity of  $L^2[0,1]$  as a Banach space. Thus,  $L^{\wedge \wedge} \approx  L^2[0,1]^{\wedge \wedge} \approx  L^2[0,1]$. 
	  
	 \item[ (2)]  $L$ is metrizable and complete, therefore its  dual group $L^\wedge$ endowed with the compact-open topology $c(L^\wedge, L)$  is a $k$-space (\cite{Auss0}, \cite{CH}). 
	 \item [(3)]  Every $\sigma (L^\wedge, L)$-compact subset of $L^\wedge$ is equicontinuous with respect to $L$.\\
	  (For a proof see \cite{CMT}.  The term $g$-barrelled defines  this property, see Section \ref{main}, \ref{def-g-barr}).
	 \item [(4)]  The natural mapping 
	   $L \to  (L^\wedge, \sigma (L^\wedge, L))^\wedge  $ is a topological isomorphism.\\ (This derives from (3) plus the local quasi-convexity of $L$.  A direct  proof  can be  seen  in \cite{BCMT}.)
	   \item [(5)]  $c(L^\wedge,L)$  and $\sigma (L^\wedge, L)$ induce the same topology  on  any  $M \subset L^\wedge$ which is equicontinuous with respect to $L$. (This fact is  well known).
	  \end{itemize}

	\begin{claim}
		$ G: = (L^\wedge, \sigma (L^\wedge, L))$  is a precompact group which is not a $k_\T$-group. 
	\end{claim}
{\bf Proof.} To show that $G$ is not a $k_\T$-group, we consider 	the evaluation mapping corresponding to an element  $x \in L^2[0,1] \setminus L$. If  $\tilde{x}:  L^\wedge \to \T$ denotes the evaluation, $\tilde{x}$ is a continuous character  with respect to  $ c(L^\wedge, L)$, therefore $\tilde{x}_{|H}$ is  continuous 	for  any subset $H\subset L^\wedge$.
 Let  $K \subset L^\wedge$   be $\sigma (L^\wedge, L)$-compact. By (3) $K$ is equicontinuous, and by (5)  $\tilde{x}_{|K}$  is $\sigma (L^\wedge, L)_{|K}$-continuous. However,   $\tilde{x} $ is not $ \sigma (L^\wedge, L)$-continuous. Observe that  the  continuous characters on $(L^\wedge, \sigma (L^\wedge, L) )$ are  the evaluation mappings on elements of $L$ and $x \notin L$.    This proves   Claim 1.

	 \begin{claim}
		  The  $k_\T$-extension of $\sigma (L^\wedge, L) $  is precisely $\sigma (L^\wedge, L^{\wedge \wedge})$. Thus, the group $(L^\wedge, \sigma (L^\wedge, L^{\wedge \wedge}))^\wedge $ is complete.
	 \end{claim}  
 
{\bf Proof.} 
	 Fix a character  $f: (L^\wedge, \sigma (L^\wedge, L) ) \to \T $   such that $f_{|K} $ is continuous for every $\sigma (L^\wedge, L) $-compact $K \subset L^\wedge$. Since $\sigma (L^\wedge, L) \le c(L^\wedge, L)$, 
	   $f_{|K} $ is  continuous with respect to $c(L^\wedge, L)_{|K} $. By (2),  $f$ is continuous with respect to the compact-open topology  of $L^\wedge$,  therefore  $f \in L^{\wedge \wedge}$.\\
	 On the other hand,  from $ L^{\wedge \wedge} \approx L^2[0,1]^{\wedge \wedge}$ we deduce that every element in $ L^{\wedge \wedge}$ has the form $ \tilde{x} $ for some $x \in L^2[0,1]$.
	 Thus, the   $k_{\T}$-extension of $\sigma (L^\wedge, L) $ is  $\sigma (L^\wedge, L^{\wedge \wedge})$.
	   
	   The completeness of ($L^\wedge, \sigma(L^\wedge, L^{\wedge \wedge}))^\wedge $ is proved by Proposition \ref{complecion}.
	 
	\begin{claim}
 $(L^\wedge, \sigma (L^\wedge, L^{\wedge \wedge}))^\wedge$ is not the completion of 	    $G^\wedge = (L^\wedge, \sigma (L^\wedge, L))^\wedge$. 
	\end{claim} 
 {\bf Proof.} By (4) $L$ is topologically isomorphic to $(L^\wedge, \sigma (L^\wedge, L))^\wedge$ and by  (2) $L$ is complete. Thus, $(L^\wedge, \sigma (L^\wedge, L))^\wedge$ is a complete proper subgroup of  $(L^\wedge, \sigma (L^\wedge, L^{\wedge \wedge}))^\wedge$.
	  This  proves the last statement of Proposition \ref{complecion}.} 
	\end{ex}

\section{On the existence of  $g$-barrelled topologies in a group duality}\label{main}

 The $g$-barrelled groups were introduced in \cite{CMT}.
They constitute  a class of abelian topological groups    which   is, in some sense,  the counterpart of the class of  barrelled spaces,  well-known objects in the theory of locally convex spaces.\\
 Next we give   some basic properties of $g$-barrelled groups, and  convenient notation to deal with them.

For a    topological group $(G, \tau)$,  the symbol  $\mathcal K$ will stand for  the family of all $\sigma (G^\wedge, G)$-compact subsets of $G^\wedge$.
	 The topology on $G$ of uniform convergence on the members of  $\mathcal K$ will be denoted by $\tau_{\mathcal K}$.   The family   $\{ K^\triangleleft,~ K \in \mathcal{K} \} $  describes a basis of zero-neighborhoods for the topology $\tau_{\mathcal K}$  on $G$. Having a basis of quasi-convex sets,  $\tau_{\mathcal K}$   is a locally quasi-convex topology. Any locally quasi-convex topology $\nu$ on a group $G$ is totally determined by the family $\mathcal E$ of all the  equicontinuous subsets that it produces in its dual group $(G, \nu)^\wedge$. More precisely, $\nu$ is the topology of uniform convergence on the sets of $\mathcal E$,  and    $\{ L^\triangleleft,  ~L \in \mathcal E \}$ is a basis of zero neighborhoods for $\nu$. 
	  A thorough study of this topic is done in \cite{JME}.

  \begin{defn}\label{def-g-barr} \cite{CMT}
  {\em 	A topological group $(G, \tau) $ is } $g$-barrelled {\em if every $\sigma (G^\wedge, G)$-compact subset  of $G^\wedge$ is equicontinuous with respect to $\tau$.
  	 The term  $g$-barrelled also applies to the topology $\tau $.}
  \end{defn}

  The class of $g$-barrelled groups  contains  the following subclasses of abelian  topological groups:   locally compact,  complete metrizable,  pseudocompact,  locally pseudocompact,  precompact Baire bounded torsion (see \cite{CMT}, \cite{HM}, \cite{DMT}, \cite{CDT}).

\begin{rems}\label{eq} 
	\begin{itemize}
		\item [(i)] 	{\em    Local quasi-convexity is not required in the   definition of a $g$-barrelled group. Nevertheless, if $(G, \tau)$ is a $g$-barrelled group,  and  $\mathcal Q \tau$ is the locally quasi-convex modification of $\tau$,  then $(G, \tau)^\wedge = (G, \mathcal Q \tau)^\wedge $ and   $(G, \mathcal Q \tau)$ is $g$-barrelled and locally quasi-convex.}
		\item[(ii)]{\em There is at most one $g$-barrelled locally quasi-convex topology on a topological group $G$ which is  compatible with the duality $(G, G^\wedge)$. }
	\end{itemize}
\end{rems}

The proof of (i) follows   from the fact that the equicontinuous subsets with respect to  $\tau$ and  $\mathcal Q \tau  $ coincide (\cite[Proposition 7.1]{JME}). The statement  (ii) has an straightforward proof.

 An interesting feature of locally quasi-convex $g$-barrelled groups is  that they are topologically isomorphic to duals of precompact groups. For further use we express this property as a lemma.
 \begin{lem}\cite[2.6]{BCDM}\label{L1}
 	Let $(G, \tau)$ be a locally quasi-convex,  $g$-barrelled group. Then,  the natural evaluation  mapping  $ e: (G, \tau) \to (G^\wedge, \sigma ( G^\wedge, G))^\wedge$ is a topological isomorphism.
 	
 \end{lem}

As we  explained in the introduction,  there are group dualities without $g$-barrelled topologies.
Now the question is to find  conditions on a topological group $(G, \tau)$ or in its dual $G^\wedge$,  which imply the existence of $g$-barrelled topologies on $G$  compatible with $\tau$. By the remark (i) in \ref{eq}, the question may be equivalently  reformulated as follows:  under which conditions is there a $g$-barrelled locally quasi-convex topology  in a fixed  duality  $(G, G^\wedge)$?

  \begin{pro} \label{bar1}
  	Let $(G, \tau)$ be a MAP topological group.
  	There exists a $g$-barrelled locally quasi-convex  topology on $G$ in the dual pair $(G, G^\wedge)$ if and only if  $\tau_{\mathcal K}$ is compatible with $\tau$. 
\end{pro}
\begin{proof}
 Observe first that all the topologies compatible with $\tau$ produce, in the common dual group $G^\wedge$, the same family of $\sigma (G^\wedge, G)$-compact subsets as $\tau$. In other words, the family  $\mathcal K$  is an invariant of the duality $(G, G^\wedge)$.
Assume now that $\tau_{\mathcal{K} }$ is compatible with $\tau$. Clearly, every $K \in \mathcal{K}$ is equicontinuous
	 with respect to $\tau_{\mathcal{K}}$. So  $\tau_{\mathcal K}$ is the unique   $g$-barrelled, locally quasi-convex topology on $G$ which is in the dual pair $(G, G^\wedge)$.

For the converse implication, we prove first that the existence of a $g$-barrelled locally quasi-convex topology $\mu$  on $G$ compatible with the pair $(G, G^\wedge)$ implies $\mu = \tau_{\mathcal K}$. 
Assume  $\mu$ meets the mentioned requirements. Choose $V \subset G$ a quasi-convex  neighborhood of zero in $\mu$. Then  $V^\triangleright$ is $\sigma (G^\wedge, G) $-compact, and therefore
$V^{\triangleright \triangleleft}$ is a neighborhood of zero.  From $V^{\triangleright \triangleleft} = V$, we obtain $\mu \leq \tau_{\mathcal K}$.
For the converse inequality, fix  $K \in \mathcal K$. Since $ (G, \mu)$ is $g$-barrelled, $K$   is equicontinuous with respect to $\mu$. Thus, it exists a $\mu$-neighborhood of zero $W$ such that $W \subset K^\triangleleft $. This implies that 
$\tau_{\mathcal K} \leq \mu$.\
 \end{proof}

 By  remark (i) in \ref{eq}, if $\tau_{\mathcal K}$ is not compatible, there are no $g$-barrelled topologies in the dual pair. Next we give a necessary and sufficient condition for  $\tau_{\mathcal K}$ to be  compatible with $\tau$, or equivalently,  to be
 compatible with $\sigma(G, G^\wedge)$.

\begin{pro}\label{bar2}
	Let $(G,  \tau)$ be an abelian MAP topological group.
	 The following statements are equivalent:
	\begin{itemize}
		\item[(1)]
		$\tau_{\mathcal K}$ is compatible with $\sigma  (G, G^\wedge)$, that is $(G, \tau_{\mathcal{K}})^\wedge  = G^\wedge$.

\item[(2)] The group $(G^\wedge, \sigma (G^\wedge, G))$ is semireflexive.
	\end{itemize}
\end{pro}
\begin{proof}
	
	The proof is an easy consequence of the following argument. By Comfort-Ross Theorem,  $(G^\wedge, \sigma (G^\wedge, G))^\wedge$ can be algebraically identified with  $G$ by means of the  evaluation mapping $e : G \to (G^\wedge, \sigma(G^\wedge, G))^\wedge$, defined by $ x \mapsto \widetilde{ x}: \phi \mapsto \phi(x)$.   On the other hand the  topology for the Pontryagin dual $(G^\wedge, \sigma(G^\wedge, G))^\wedge$ is the topology of uniform convergence on  the $\sigma (G^\wedge, G)$-compact subsets of $G^\wedge$. A zero neighborhood basis for $(G^\wedge, \sigma(G^\wedge, G))^\wedge$  is given by the family $\{ K^\triangleright, ~ K \in \mathcal K \}$, whilst a zero  neighborhood basis for $ (G, \tau_{\mathcal K} )$ is given by the family $\{K^\triangleleft, ~K \in \mathcal K \}$.    The direct and inverse polars 
	 can be identified since $e$ is bijective and  $e (K^\triangleleft) = K^\triangleright$.   Thus, $e$ is a topological isomorphism:
	$$(G, \tau_{\mathcal K}) \stackrel{e}{\approx} (G^\wedge, \sigma(G^\wedge, G))^\wedge$$
	Taking now duals in both sides we obtain:
	$$(G, \tau_{\mathcal K})^\wedge \approx (G^\wedge, \sigma(G^\wedge, G))^{\wedge \wedge } $$
In order to prove that	(1) $\Rightarrow$ (2), assume that  $\tau_\mathcal{K}$ is compatible with $\tau$, that is $(G, \tau_{\mathcal K})^\wedge = G^\wedge$. From the last isomorphism, it follows that $(G^\wedge, \sigma(G^\wedge, G))$ is semireflexive.\\
The implication 	(2)$\Rightarrow$ (1)   also follows from the mentioned isomorphism.
	\end{proof}

The results of Propositions \ref{bar1} and \ref{bar2} yield the following:
\begin{thm}\label{existencia}
	Let $(G, \tau)$ be a MAP topological group. The following assertions are equivalent:
	\begin{itemize}
	 	\item [(1)]
		There exists a $g$-barrelled topology in the duality $(G, G^\wedge)$.
		\item[(2)]
		 $(G^\wedge, \sigma (G^\wedge, G))$ is semireflexive.
		 \item[(3)]
		 The topology $\tau_{\mathcal K}$ on $G$ is compatible with $\tau$.
	\end{itemize}
\end{thm}

For  a MAP topological group $(G, \tau)$ which does not satisfy the conditions of the preceding theorem,
  it is  natural to ask if $(G, \tau_{\mathcal K})$ can  still be  $g$-barrelled  in the new duality it generates.  We  provide below a result in this line.
  First, recall  the notion of {\it determined subgroup}.
\begin{defn}
	{\em  A  subgroup $Y$ of an abelian topological group  $(X, \tau)$  is said {\it to determine} $X$ if the inclusion  $i: (Y, \tau_{|Y}) \to (X,\tau)$ has a dual mapping $i^\wedge :(X,\tau)^\wedge \to (Y, \tau_{|Y})^\wedge$ which is a topological isomorphism. It is  frequent to call $Y$ a {\it determined subgroup of $X$}}.
\end{defn}
The  above mentioned dual goups carry the compact-open topology, in other words   they are  Pontryagin  duals. Clearly,  the restriction mapping  $i^\wedge$ is continuous without additional conditions on $X$ or $Y$. If $Y$ is dense in $X$, then $i^\wedge$ is monomorphism.  Thus, the only specific property to be a determined subgroup is that the mapping  $i^\wedge$ must be open. This is achieved if for every compact set $K \subset X$ there is a compact set $L \subset Y$ such that $i^\wedge (K ^\triangleright) \supset L^\triangleright$. In what follows we relax this expression and simply say that $L^\triangleright \subset K^\triangleright$, which permits also to say   that the compact-open topologies in $X^\wedge$ and $Y^\wedge$ coincide.
\begin{thm} \label{determined}
 Let $(G,\tau)$ be a MAP topological group, and let
$\mathcal J : =  (G, \tau_{\mathcal{K}})^\wedge$.  The group $(G, \tau_{\mathcal K})$ is $g$-barrelled iff $(G^\wedge,\sigma (G^\wedge, G))  $ determines $(\mathcal J, \sigma (\mathcal J, G))$.
\end{thm}

\begin{proof}
 Clearly $G^\wedge \subset \mathcal{J} \subset Hom (G, \T)$ and   $(G^\wedge,\sigma (G^\wedge, G))  $ is a topological subgroup of $(\mathcal J, \sigma (\mathcal J, G))$. Since $G$ is MAP,  $G^\wedge$ is dense in $Hom_p (G, \T)$. 
 Therefore $G^\wedge$ is also dense in $(\mathcal J, \sigma (\mathcal J, G))$ and their   dual groups  can be algebraically identified, which we  simply write as an equality:
	$$(G^\wedge,\sigma (G^\wedge, G))^\wedge = (\mathcal J, \sigma (\mathcal J, G))^\wedge $$
	
$\Rightarrow$)	Assume  that $(G, \tau_{\mathcal K})$ is $g$-barrelled.  We must prove that the compact-open topology in 	 $(G^\wedge,\sigma (G^\wedge, G))^\wedge$ and in $ (\mathcal J, \sigma (\mathcal J, G))^\wedge$ coincide (the underlying set of both of them can be identified to $G$).
	
	To this end, fix  $K \subset \mathcal J$  a  $\sigma (\mathcal J, G)$-compact subset. We must find a $\sigma (G^\wedge, G)$-compact subset  $L \subset G^\wedge$, such that $L^\triangleright \subset K^\triangleright$. Since $(G, \tau_{\mathcal K})$ is $g$-barrelled, there  exists a $\tau_{\mathcal K}$-zero neighborhood  $V$ such that $K \subset V^\blacktriangleright$ (the black triangle symbol indicates that the  polar is taken in $\mathcal J$).  By the definition of $\tau_{\mathcal K}$,  $V \supset L^\triangleleft$  for some $L \subset G^\wedge$ which is $\sigma (G^\wedge, G)$-compact.
	    Thus $K \subset L^{\triangleleft \blacktriangleright}$, and
	    taking    polars on both sides
	     we obtain:  $K^\triangleright \supset ( L^{\triangleleft \blacktriangleright})^\triangleright $.
	
	    On the other hand $L^\triangleleft \subset G$ is quasi-convex in $\tau_{\mathcal K}$, therefore $L^\triangleleft = (L^\triangleleft)^{\blacktriangleright \triangleleft} = ( L^{\triangleleft \blacktriangleright})^\triangleleft $.

	 Implementing  this in the above expression, we get:
	
	    $$K^\triangleright \supset ( L^{\triangleleft \blacktriangleright})^\triangleright = ( L^{\triangleleft})^{ \blacktriangleright \triangleleft}  = L^\triangleleft$$
	
Finally  the inverse polar $L^\triangleleft$ can be identified with $L^\triangleright$, since $(G^\wedge, \sigma(G^\wedge, G))^\wedge = G$. Thus, we can simply write   $K^\triangleright \supset L^\triangleright$, which  proves that $(G^\wedge,\sigma (G^\wedge, G))  $ determines $(\mathcal J, \sigma (\mathcal J, G))$.

$\Leftarrow)$   In order to prove that $(G, \tau_{\mathcal K})$ is $g$-barrelled,  fix now a $\sigma (\mathcal J, G)$-compact subset  $K$ of $\mathcal J$. By the assumption,  
there exists a $\sigma( G^\wedge, G)$-compact   $L \subset G^\wedge$ such that  $K^\triangleright \supset L^\triangleright$. Here the polars are taken in  $(\mathcal J, \sigma (\mathcal J, G))^\wedge $ and $(G^\wedge,\sigma (G^\wedge, G))^\wedge$ respectively, but both dual groups are identified, as said above. Taking inverse polars with respect to $\mathcal J$ we can write:
$K^{\triangleright\triangleleft} \subset L^{\triangleright\triangleleft}$,
Thus: $$ K \subset K^{\triangleright\triangleleft} \subset L^{\triangleright\triangleleft} \subset L^{\triangleleft \triangleright}$$
Since $L^\triangleleft$ is a neighborhood of zero in $\tau_{\mathcal K}$,  $K$ is equicontinuous and therefore $(G, \tau_{\mathcal K})$ is $g$-barrelled.
	\end{proof}
Concerning the last theorem, it arises the question whether  the claim ``$(G^\wedge,\sigma (G^\wedge, G))  $ determines $(\mathcal J, \sigma (\mathcal J, G))$" is always true. Thus, we leave an open problem:
\begin{q}
{\em	Give an example of a topological group $(G, \tau)$ such that
	 $(G^\wedge,\sigma (G^\wedge, G))  $ does not determine $(\mathcal J, \sigma (\mathcal J, G))$, where $\mathcal J = (G, \tau_{\mathcal K})^\wedge$.}
\end{q}

   Denote by $\mathcal B$  the class of all precompact Hausdorff abelian groups  which are the  Bohr reflection of a locally compact group,   as in  \cite{HT}. Explicitly, $ (G, w) \in \mathcal B $ if there exists a locally compact group topology $\tau$ on $G$  such that $\tau^+ = w$. We end this section with two results which might complement the contents of \cite{HT}. The first one provides  examples of groups which are not in $\mathcal B$.  Loosely speaking, if a topological group $(G, \tau)$ gives rise to a duality which contains a locally quasi-convex $g$-barrelled non locally compact topology, then $(G, \tau^+)$ is not in $\mathcal B$.

\begin{pro} \label{lc}
	Let $(G, \tau)$ be a topological group  such  that $\tau_\mathcal{K}$  
	is compatible with $\tau$.  If $(G, \tau_{\mathcal{K}})$ is not locally compact, then $(G, \tau^+) \notin \mathcal{B}$.
	
\end{pro}
\begin{proof}
This is an easy consequence of the uniqueness of a $g$-barrelled  locally quasi-convex topology on $G$ compatible with  the duality $(G, G^\wedge)$. The topology $\tau_{\mathcal K}$ already meets these requirements. If $G$ could be equipped with a locally compact topology $\mu$, then $(G, \mu) $ would be a $g$-barrelled, locally quasi-convex group. Therefore $\mu$ cannot be compatible with $\tau$. Thus,  $\mu^+ \neq \tau^+$ and  $(G, \tau^+) \notin \mathcal{B}$.
\end{proof}
\begin{pro}
	The class $\mathcal B$ is not (countably) productive.
	
\end{pro}
\begin{proof}
	Take a family $\{ (G_i, w_i) \in \mathcal{B}, ~  i\in I \}$, whose members are non-compact and $|I| \geq \aleph_0$. By the definition of $\mathcal B$, for every $i \in I$ there exists  a locally compact topology $\tau_i $ in $G_i$ such that $(G_i, \tau_i^+)   = (G_i, w_i)$.
	Observe that the product $ G: = \prod (G_i, w_i)$ is a precompact Hausdorff  group. Further, the product topology $\prod w_i$ is the minimum of all the locally quasi-convex  topologies  on $G$ compatible with  the duality $(G, G^\wedge)$.\\
	 Since the product of $g$-barrelled groups  is also $g$-barrelled (\cite[3.4]{BCDM} ), $\prod  \tau_i$ is a $g$-barrelled locally quasi-convex topology in the duality $(G, G^\wedge)$.
	 Clearly (even if $I$ is countable ) $\prod  \tau_i$ is not locally compact, and Proposition \ref{lc} applies.
\end{proof}


 \section{ The family  $\mathcal D_G$  of  compatible  topologies on a discrete group G}

Let  $G$ be a group,  $\delta$ the discrete topology on $G$   and    $\mathcal{D}_G$    the family of all group topologies on $G$ compatible with $\delta$.
 All the  elements  in $\mathcal{D}_G$ lie between $\delta^+$ and $\delta$ and  have $Hom (G, \T)$ as character group.   Glicksberg Theorem applied to $(G, \delta)$ yields that   any topology $\tau \in \mathcal{D}_G$ has the same family of compact subsets as $\delta$.
 Thus,  the topologies in $\mathcal{D}_G$ give rise to  the same dual group, algebraically and topologically.\\ Let us write:  $(G, \tau)^\wedge = Hom_p(G, \T)$, for all $\tau \in \mathcal{D}_G$.

 We characterize now the family $ \mathcal{D}_G$ in the class of   MAP topological groups. 

 \begin{pro}\label{DD}
Let $(G, \tau)$ be a MAP group. The following statements are equivalent:
\begin{itemize}
	\item[(1)] $\tau \in  \mathcal D_G$.
	\item[(2)] $(G, \tau)^\wedge$ is  a  compact group and the $\tau$-compact subsets of $G$ are finite.
	\item[(3)] $ (G^\wedge, \sigma (G^\wedge, G))$ is a compact group.
\end{itemize} 	
\end{pro}
\begin{proof}
	
(1)	$\Rightarrow$ (2).
 As said in the preceding comments, the $\tau$-compact subsets of $G$ are finite and  $(G, \tau)^\wedge = Hom_p(G, \T)$.  
  Since  $Hom_p(G, \T)$ carries the pointwise convergence topology, it is a  closed subgroup in the product $\T^G$. Therefore $Hom_p(G, \T)$ is a compact Hausdorff  group.

   (2) $\Rightarrow $ (3).
    Clearly, if the $\tau$-compact subsets are finite,  the compact-open topology in $G^\wedge$  coincides with the pointwise convergence topology, thus $ (G^\wedge, \sigma (G^\wedge, G)) = (G, \tau)^\wedge$ is compact. 

(3)	$\Rightarrow $ (1).  Observe that
  $ (G^\wedge, \sigma (G^\wedge, G))$  is a closed   subgroup of the compact group $ Hom_p(G, \T) $.
 Assume by contradiction that   $G^\wedge \neq Hom (G, \T) $. Then,   there  exists a non-null continuous  character on  $Hom_p(G, \T)$ which is null in $G^\wedge$. Since  the continuous characters on  $Hom_p(G, \T)$ are precisely the evaluations on points of $G$, there must exist a non null $x \in G$ such that $\phi (x) = 1$ for all $\phi \in G^\wedge$. This contradicts the fact that $(G, \tau)$ is MAP.
Therefore, it must be $G^\wedge = Hom (G, \T)$ which proves (1).
	\end{proof}

\begin{cor} On an abstract group $G$, the family   $\mathcal D_G$ does not contain any nondiscrete $k$-group  topology. In particular, every metrizable nondiscrete group has discontinuous characters.
	
\end{cor}

\begin{proof}
	Since all the     topologies compatible with $\delta$  give rise to the same  family of compact subsets, there is at most one  $k$-group topology in $\mathcal D_G$. On the other hand, $\delta$ is already a $k$-group topology in $\mathcal D_G$. Thus the first claim is proved.\\
	If   $\mu$ is a metrizable nondiscrete  group topology  on $G$, it   
is a $k$-group topology. Therefore
  $\mu \notin \mathcal {D}_G$, which means that
     $(G, \mu)^\wedge \neq Hom(G, \T)$.
	\end{proof}
\begin{cor} 
	Every nondiscrete Mackey group (or $g$-barrelled group)  $G$ admits non-continuous characters.
\end{cor}
\begin{proof} Let $(G, \mu) $ be a Mackey group. Then $ (G, \mu)^\wedge \neq Hom (G, \T)$, for otherwise $\tau $ would be compatible with $\delta$ and by the assumption $\mu = \delta$. 
	If $(G, \mu)$ is $g$-barrelled, the previous argument  can be applied to $(G, \mathcal{Q} \mu)$, which  by  Remark \ref{eq} (i)   is a  Mackey group, and admits the same character group as $(G, \mu) $. 
\end{proof}	

 The property of ``having finite compact subsets" is not sufficient to characterize the  elements of   $\mathcal {D}_G$ in the class of MAP topological groups. The next statement gives a related feature.
\begin{pro} \label{semi}
Let  $(G, \tau)$ be a topological group whose compact subsets are finite.  Then, $(G, \tau)$ is semireflexive.
\end{pro}
\begin{proof}
	 The Pontryagin  dual of $(G, \tau)$ is   $ (G^\wedge, \sigma ( G^\wedge, G)  )$.
Since the character group  of $(G^\wedge, \sigma ( G^\wedge, G)  )$ is algebraically  isomorphic to $G$, we have that  $G^{\wedge \wedge}$ and $G$ are isomorphic as groups.
\end{proof}

The more restrictive assumption that the
 $\sigma (G, G^\wedge)$-compact subsets of $G$ are finite, yields $g$-barrelledness,  as specified next:
\begin{pro} \label{DB}
		Let   $(G, \tau)$ be a topological group. The following statements are equivalent:
	\begin{itemize}
	\item[(1)] The   $\sigma(G, G^\wedge)$-compact subsets of $G$ are finite.
	\item[(2)] The Pontryagin dual of $(G, \tau)$ coincides with $  (G^\wedge, \sigma (G^\wedge, G))$ and it is $g$-barrelled.
\end{itemize} 	
\end{pro}
\begin{proof}
	(1) $ \Rightarrow $ (2). Since every $\tau$-compact subset of $G$ is $\sigma (G, G^\wedge)$-compact, as in the proof of \ref{semi},  we obtain that   $ (G^\wedge, \sigma (G^\wedge, G))$ is the Pontryagin dual of $(G, \tau)$. 
	 In order to prove that  $X: = (G^\wedge, \sigma (G^\wedge, G))$ is $g$-barrelled, fix
$K \subset X^\wedge$  compact  with respect to $\sigma (X^\wedge, X)$. Take into account that 	$X^\wedge$ and $G$ are algebraically isomorphic  and the topology $\sigma (X^\wedge, X)$  can be identified with $ \sigma (G, G^\wedge)$. Thus, $K$ can be considered a   $\sigma (G, G^\wedge)$-compact subset of $G$. By (1) $K$ is finite and therefore equicontinuous.

(2) $\Rightarrow$ (1).
Fix a $\sigma (G, G^\wedge)$-compact subset  $L \subset G$.
By the above mentioned  identifications,
 $L$ can be considered as a $\sigma (X^\wedge, X)$-compact subset of $X^\wedge$.
Since  $ X =  (G^\wedge, \sigma (G^\wedge, G))$ is $g$-barrelled, $L$ is equicontinuous with respect to $ \sigma (G^\wedge, G) $, which is a precompact topology. Therefore $L$ is finite.

\end{proof}

If $\tau \in \mathcal D_G$, clearly $\tau$ satisfies (1) and (2) in Proposition  \ref{DB}.  The converse does not hold as the next example shows.
\begin{ex}\label{no pseudocpt}
		{\bf A  topological group $G$ which is not in $\mathcal D_G$ and  the  $\sigma (G, G^\wedge)$-compact subsets are finite.}
		
			{\em Let $\mathcal L$ be a second category subgroup of $\T$ and consider $\T$ as the group of characters on $\Z$.
		  If $G : = (\Z, \sigma(\Z, \mathcal L))$, clearly $G$ is a precompact Hausdorff group
		 such that  $G^\wedge = \mathcal L$. The $\sigma (G^\wedge, G)$-topology on $\mathcal L$ coincides wih the topology induced on $\mathcal L$ as a subspace of $\T$. Thus, $\mathcal L$ is separable and   $  (G^\wedge, \sigma (G^\wedge, G))$ is $g$-barrelled (by \cite[1.6]{CMT}). Since  $G$ satisfies (2) in Proposition \ref{DB},  the $\sigma (\Z, \mathcal L)$-compact subsets of $\Z$ are finite. Obviously, $\sigma(\Z, \mathcal L) \notin \mathcal D_G$. }

		 {\em Observe that $G$ is not reflexive: its dual group is $\mathcal L$ and the bidual $G^{\wedge \wedge}$ is algebraically isomorphic to  $\Z$, but the compact-open  topology on $G^{\wedge \wedge}$ is discrete. In fact, being $\mathcal L$ dense in the metrizable compact group $\T$, it has the same dual as $\T$ algebraically and topologically. Since $G$ is non-discrete, it is not topologically isomorphic to  $G^{\wedge \wedge}$.
		
		  The starting group $G$ is not $g$-barrelled, because it is  countable and  nondiscrete. Its Pontryagin dual $  (G^\wedge, \sigma (G^\wedge, G))$ is $g$-barrelled, metrizable, noncompact, neither countably compact. }
\end{ex}



\begin{thebibliography}{widestlabel}
	\bibitem{Auss0}  L.  Aussenhofer.
	{\em Contributions to the duality theory of abelian topological groups and to the theory of nuclear groups}. Dissertationes Mathematicae 384 (1999).
	\bibitem{Auss1}  L. Aussenhofer.
	{\em On the non-existence of the Mackey topology for locally quasi-convex groups}. Forum Math. Vol. 30,5 (2018), 1119-1128.
	\bibitem{zn} L. Ausenhofer. \emph{The group $\Z^{(\N)}$ has no Mackey topology}.  Topology Appl. 281 (2020).

	\bibitem{BCDM} T. Borsich,  M. J. Chasco, X. Dom\'inguez, E. Mart\'in-Peinador. {\em On g-barrelled groups and their permanence properties.} J. Math. Anal. Appl. 473 (2019), no. 2, 1203--1214.
	
	\bibitem{TM} M. Bruguera. {\em Grupos topol\'ogicos y grupos de convergencia: estudio de la dualidad de Pontryagin.} Doctoral Dissertation, Barcelona, 1999.
	
	\bibitem{BCMT} M. Bruguera, M.J. Chasco, E. Mart\'in-Peinador and V. Tarieladze. {\em Completeness properties of locally quasi-convex groups.}
	Topology Appl. 111 (2001), 81--93.
	\bibitem{CMT}
	M. J. Chasco, E. Mart\'in-Peinador, V. Tarieladze. {\em  On Mackey topology for groups.} Studia Math. 132 (1999), no. 3, 257--284. MR1669662.
	
	\bibitem{CH} M. J. Chasco. {\em Pontryagin duality for metrizable groups.} Arch. Math. (Basel) 70 (1998), no. 1, 22-28.
	\bibitem {CDT} M.J. Chasco, X. Dom\'inguez, M. Tkachenko. {\em Duality properties of bounded torsion topological abelian groups}. JMAA
	{\bf 448},  (2) (2017),  968--981.


	\bibitem{CTW}
	W.W. Comfort, F.J. Trigos-Arrieta and Ta-Sun Wu. \emph{The Bohr compactification, modulo a metrizable group}.
	Fund. Math. 143 (1993), no. 2, 119--136.
	
	
	\bibitem{JME}  J. M. D\'iaz Nieto, E. Mart\'in-Peinador. \textit{Characteristics of the Mackey topology for abelian topological groups}. Chapter 7 in Descriptive Topology and Functional Analysis, Eds : Juan Carlos Ferrando and Manuel L\'opez Pellicer, Springer Proceedings in Mathematics and Statistics 80  pp. 119--144 (2014).
	
	
	\bibitem{DiMT}
	D. Dikranjan, E. Mart\'in-Peinador and V. Tarieladze. {\em  Group valued null sequences and metrizable non-Mackey groups}. Forum Math. 26 (2014), no. 3, 723--757.
	\bibitem{DMT} X. Dom\'inguez, E.  Mart\'in-Peinador, V. Tarieladze.{\em  On ultrabarrelled spaces, their group analogs and Baire spaces.}
	Descriptive Topology and Functional Analysis. II, 77--87, Springer Proc. Math. Stat., 286, Springer, Cham, 2019
	
	
	\bibitem{G} S. Gabriyelyan. {\em A locally quasi-convex abelian group without a Mackey group topology}. Proc. Amer. Math. Soc. {\bf 146} (2018), no. 8, 3627--3632.
	\bibitem{JG}
	J. Galindo. {\emph Totally bounded group topologies that are Bohr topologies of LCA groups}.
	Topology Proc. 28 (2004), no. 2, 467--478. 	
	
	\bibitem{HM} S. Hern\'andez, S. Macario. {\em Dual properties in totally bounded Abelian groups.} Arch. Math. 80 (2003) 271--283
	
	
	
	\bibitem{HT} S. Hern\'andez, F. Trigos-Arrieta. {\em When a totally bounded group topology is the Bohr topology of a LCA group.}
	Topology Appl. 259 (2019), 110--123.
	

	
	\bibitem{kelley} J. L.  Kelley. {\em General Topology.}  D. Van Nostrand Company, Inc., Toronto-New York-London, 1955.
	
	\bibitem{noble}  N. Noble. {\em k-groups and duality}. Trans. Amer. Math. Soc. 151 (1970),  551-561.
	\bibitem{N2}  N. Noble. {\em The continuity of functions on cartesian products}.
	Trans. Amer. Math. Soc. 149 (1970), 187--198
	\bibitem{RT} D. Remus, F. Trigos-Arrieta. {\em Abelian groups which satsfy Pontryagin duality need not respect compactness.} Proc Amer Math Soc, (1993) {\bf 117} no. 8.
	 \bibitem{smith}  Smith, M, F.
	{\em  The Pontrjagin duality theorem in linear spaces.}
	Ann. of Math. (2) 56 (1952), 248–253.
\end{thebibliography}
\end{document}